\documentclass[11pt,reqno]{amsart}
\usepackage[utf8]{inputenc}

\usepackage{graphicx}
\usepackage{amsmath}
\usepackage{amsthm}
\usepackage{amssymb}
\usepackage{mathrsfs}
\usepackage{siunitx}
\usepackage{csvsimple}
\usepackage{tabularx}
\usepackage{booktabs}
\usepackage{bbm}
\usepackage{tikz-cd}
\usepackage{quiver}
\usepackage{adjustbox}
\usepackage{enumitem}
\usepackage{stmaryrd}
\usepackage[plainpages=false]{hyperref}
\usepackage{faktor}
\makeatletter
\DeclareRobustCommand*{\mfaktor}[3][]
{
   { \mathpalette{\mfaktor@impl@}{{#1}{#2}{#3}} }
}
\newcommand*{\mfaktor@impl@}[2]{\mfaktor@impl#1#2}
\newcommand*{\mfaktor@impl}[4]{
   \settoheight{\faktor@zaehlerhoehe}{\ensuremath{#1#2{#3}}}
   \settoheight{\faktor@nennerhoehe}{\ensuremath{#1#2{#4}}}
      \raisebox{-0.5\faktor@zaehlerhoehe}{\ensuremath{#1#2{#3}}}
      \mkern-4mu\diagdown\mkern-5mu
      \raisebox{0.5\faktor@nennerhoehe}{\ensuremath{#1#2{#4}}}
}
\makeatother
\makeatletter
\newenvironment{innerproof}[1][\proofname]
  {\par\normalfont \topsep6\p@ \@plus6\p@\relax
  \trivlist
  \item[\hskip\labelsep\itshape#1\@addpunct{.}]\ignorespaces}
  {\endtrivlist\@endpefalse}
\makeatother
\counterwithin{equation}{section}

\theoremstyle{definition}
\newtheorem{Definition}{Definition}[section]

\newtheorem{Remark}[Definition]{Remark}
\newtheorem{Example}[Definition]{Example}
\newtheorem{Warning}[Definition]{Warning}

\theoremstyle{plain}
\newtheorem{Theorem}[Definition]{Theorem}
\newtheorem{Lemma}[Definition]{Lemma}
\newtheorem{Proposition}[Definition]{Proposition}
\newtheorem{Corollary}[Definition]{Corollary}

\newtheorem*{Claim}{Claim}

\DeclareMathOperator{\Hom}{Hom}

\DeclareMathOperator{\End}{End}

\DeclareMathOperator{\Image}{Im}

\DeclareMathOperator{\degree}{deg}

\DeclareMathOperator{\Rep}{Rep}
\DeclareMathOperator{\Vect}{Vect}

\DeclareMathOperator{\Kspan}{span}

\newcommand{\identity}{\textnormal{id}}

\newcommand{\bC}{\mathbb{C}}

\newcommand{\bF}{\mathbb{F}}

\newcommand{\bL}{\mathbb{L}}

\newcommand{\bN}{\mathbb{N}}

\newcommand{\bZ}{\mathbb{Z}}

\title{Scalar extensions of quiver representations over $\mathbb{F}_1$}
\author{Markus Kleinau}
\address{Mathematical Institute of the University of Bonn, Endenicher Allee 60, 53115 Bonn, Germany}
\email{mkleinau@math.uni-bonn.de}

\begin{document}

\begin{abstract}
    Let $V$ and $W$ be quiver representations over $\mathbb{F}_1$ and let $K$ be a field. The scalar extensions $V^K$ and $W^K$ are quiver representations over $K$ with a distinguished, very well-behaved basis. We construct a basis of $\Hom_{KQ}(V^K,W^K)$ generalising the well-known basis of the morphism spaces between string and tree modules. We use this basis to give a combinatorial characterisation of absolutely indecomposable representations. Furthermore, we show that indecomposable representations with finite nice length are absolutely indecomposable. This answers a question of Jun and Sistko.
\end{abstract}

\maketitle

\setcounter{tocdepth}{1}
\tableofcontents

\section{Introduction}
    Some important classes of representations over bound quivers admit a very well-behaved basis. Namely one where the arrows send basis elements to other basis elements or zero and no arrow sends two different basis elements to the same basis element. Examples include tree and string modules, permutation representations, interval modules over incidence algebras, projective and injective modules over monomial algebras and laminated modules over preprojective algebras. The notion of quiver representations over $\bF_1$, as introduced by Szczesny \cite{SZCZESNY}, is designed to capture the combinatorial structure of these bases.
    
    An $\bF_1$-vector space is a pointed set. The distinguished point is the zero element while the remaining elements form a basis of the vector space. The $\bF_1$-linear maps are pointed maps satisfying an injectivity condition. We let $Q$ be a quiver and $K$ be a field. Representations of $Q$ over $\bF_1$ are defined in analogy to representations over $K$. There is a scalar extension functor that turns a representation $V$ of $Q$ over $\bF_1$ into a representation $V^K$ of $Q$ over $K$. If $V$ was obtained from a basis as above then $V^K$ is isomorphic to the original representation. The category of quiver representations over $\bF_1$ behaves in many ways like a non-additive version of a module category: It has kernels, cokernels, direct sums and tensor products. In addition it satisfies versions of the Jordan-Hölder and Krull-Schmidt theorems. This was first observed by Szczesny \cite{SZCZESNY}.

    The goal of this paper is to study aspects of the relation between the combinatorics of $\bF_1$-representations and the representation theory of their scalar extensions. Specifically we will describe a basis of the space of homomorphisms between two scalar extensions of $\bF_1$-representations and characterise when all scalar extensions of such a representation are indecomposable. In particular, this answers question 6.14 in \cite{JS2}.
    
    Jun and Sistko \cite{JS1} reinterpreted the data of an $\bF_1$-representation $V$ as a labeled graph $\Gamma_V$ called the coefficient quiver of $V$. This perspective provides a powerful dictionary from representation theory to graph theory. For example, the $\bF_1$-representation $V$ is indecomposable if and only if $\Gamma_V$ is connected. Further, $V$ is simple if and only if $\Gamma_V$ is strongly connected.

    We let $V$ and $W$ be two representations of $Q$ over $\bF_1$. The first goal of the paper is to construct a basis of  $\Hom_{KQ}(V^K,W^K)$. This basis will be indexed by a subset of the direct summands of $V\otimes W$ that we call admissible components. If $\Gamma_V$ and $\Gamma_W$ are trees then this basis agrees with the one introduced by Crawley-Boevey \cite{CB} for tree modules. In particular, we obtain a combinatorial description of $\End_{KQ}(V^K)$. 

    An $\bF_1$-representation $V$ is absolutely indecomposable if $V^K$ is indecomposable for all fields $K$. Our second goal is to use the basis above to give a combinatorial characterisation of absolutely indecomposable representations. To that end, we define and study a special class of the admissible components in $\Gamma_{V\otimes V}$ that we call covering components. This yields the following theorem:
    \begin{Theorem}
        Let $V$ be an indecomposable representation of $Q$ over $\bF_1$. Then the following are equivalent:
        \begin{enumerate}[label = (\roman*)]
            \item $V$ is absolutely indecomposable,
            \item $V^\bC$ is indecomposable,
            \item $\Gamma_{V\otimes V}$ contains exactly one covering component.
        \end{enumerate}
    \end{Theorem}
    In this case, the unique covering component corresponds to the identity in our basis of $\End_{KQ}(V^K,V^K)$ for all fields $K$. This criterion can be tested efficiently in practice: All covering components can be found in $O(m^2)$ time where $m$ is the number of arrows in $\Gamma_V$.

    Absolutely indecomposable representations over finite fields were first studied by Kac \cite{Kac}, see also \cite{AbsInd}. Their main results describe how the number of absolutely indecomposable representations depend on the number of elements of the underlying field. Unfortunately, those formulas usually give the wrong result for $\bF_1$. A more detailed discussion of this can be found at the end of section 6 in \cite{JS2}.

    Jun and Sistko \cite{JS2} introduced the notion of finite nice length for representations over $\bF_1$. They showed that given a representation $V$ of finite nice length, computing the Euler characteristic of the quiver Grassmannians of $V^\bC$ becomes a combinatorial problem. To do so they adapted an approach developed by Cerulli-Irelli \cite{CI} and \cite{Haupt} to the language of $\bF_1$-representations. We apply our previous characterisation to obtain the following theorem.
    \begin{Theorem}
        Let $V$ be an indecomposable representation over $\bF_1$ with finite nice length. Then $V$ is absolutely indecomposable.
    \end{Theorem}
    
    Sections 2 and 3 recall the basic theory of quiver representations over $\bF_1$ and their coefficient quivers. Section 4 constructs the basis of the homomorphism spaces between scalar extensions. Section 5 prepares the theory of covering components which we need to prove the characterisation of absolutely indecomposable representations in section 6. The final section contains the proof that indecomposable finite nice length representations are absolutely indecomposable.
    
\section{Quiver representations over \texorpdfstring{$\bF_1$}{F1}}
    We start by recalling the category $\Rep(Q,\bF_1)$ and its basic properties following \cite{SZCZESNY} and \cite{JS2}.
    \begin{Definition}
        A \textit{vector space} over $\bF_1$ is a finite pointed set $(V,0_V)$. Let $(V,0_V)$ and $(W,0_W)$ be $\bF_1$-vector spaces. An $\bF_1$-\textit{linear map} is a pointed map $f\colon V\rightarrow W$ such that $f$ is injective on $V\backslash f^{-1}(0_W)$. This defines a category called $\Vect(\bF_1)$.
    \end{Definition}
    By abuse of notation, we will usually refer to an $\bF_1$-vector space $(V,0_V)$ just by $V$. The subset $V\backslash\{0_V\}$ should be interpreted as a basis of $V$. The category $\Vect(\bF_1)$ is not additive yet it shares many properties with the categories of vector spaces over actual fields. We adapt some of the terms from linear algebra.
    \begin{Definition}
        Let $V$ and $W$ be $\bF_1$-vector spaces and let $f\colon V\rightarrow W$ be an $\bF_1$-linear map.
        \begin{enumerate}[label = (\roman*)]
            \item The \textit{dimension} of $V$ is given by $\dim_{\bF_1}(V) = |V\backslash\{0_V\}|$.
            \item $V$ is a \textit{subspace} of $W$ if $V$ is a pointed subset of $W$.
            \item If $V$ is a subspace of $W$ then the \textit{quotient} $W/V$ is the pointed set $(W\backslash V)\cup \{0_W\}$.
            \item The \textit{kernel} of $f$ is the subspace $f^{-1}(0_W)$ of $V$. The \textit{cokernel} of $f$ is the quotient $W/\Image(f)$.
            \item There is a unique $0$-map from $V$ to $W$ given by sending all elements in $V$ to $0_W$.
            \item The \textit{direct sum} of $V$ and $W$ is the pointed set $V\oplus W =V\sqcup_{0_V\sim 0_W}W$.
            \item The \textit{tensor product}  is the pointed set $V\otimes W =V\times W /(V\times \{0_W\}\cup\{0_V\}\times W )$
            \item The \textit{dual map} $f^t\colon W\rightarrow V$ is given by 
            \[f^t(w) = \begin{cases}
                v &\textnormal{if }w \neq 0_w\textnormal{ and }f^{-1}(w) = \{v\}\\
                0_V &\textnormal{else.}
            \end{cases}\]
        \end{enumerate}
    \end{Definition}
    This gives the category $\Vect(\bF_1)$ enough structure to define exact sequences. In fact it forms a proto-exact category in the sense of \cite{Segal}. 
    \begin{Warning}
        The direct sum $V\oplus W$ is not a categorical biproduct since $\Vect(\bF_1)$ admits neither categorical products nor coproducts.
    \end{Warning}
    Next we want to consider quiver representations over $\bF_1$.
    \begin{Definition}
        A \textit{quiver} is a quadruple $Q=(Q_0,Q_1,s,t)$ where $Q_0$ is the set of vertices and $Q_1$ is the set of arrows while $s$ and $t$ are maps from $Q_1$ to $Q_0$ sending an arrow to its source, respectively its target. We always assume that $Q_0$ and $Q_1$ are finite.
    \end{Definition}
    \begin{Example}
        The $n$-loop quiver $\mathbb{L}_n$ has one vertex and $n$ loops at that vertex.
    \end{Example}
    We fix a quiver $Q$ for the rest of the paper.
    \begin{Definition}
        Let $K$ be a field or $\bF_1$. A \textit{quiver representation} of $Q$ over $K$ is a tuple $V = ((V_i)_{i\in Q_0},(f_\alpha)_{\alpha\in Q_1})$. Here, $V_i$ is a finite dimensional $K$-vector space for each $i\in Q_0$ and $f_\alpha\colon V_{s(\alpha)}\rightarrow V_{t(\alpha)}$ is a $K$-linear map for each $\alpha\in Q_1$. Let $V$ and $W$ be representations of $Q$ over $K$. A \textit{homomorphism of quiver representations} $f\colon V\rightarrow W$ is a family of maps $(f_i)_{i\in Q_0}$ where each $f_i$ goes from $V_i$ to $W_i$ and for each arrow $\alpha\in Q_1$ the following diagram commutes:
        \[\begin{tikzcd}
        	{V_{s(\alpha)}} & {W_{s(\alpha)}} \\
        	{V_{t(\alpha)}} & {W_{t(\alpha)}.}
        	\arrow["{f_{s(\alpha)}}", from=1-1, to=1-2]
        	\arrow["{W_\alpha}", from=1-2, to=2-2]
        	\arrow["{V_\alpha}"', from=1-1, to=2-1]
        	\arrow["{f_{t(\alpha)}}"', from=2-1, to=2-2]
        \end{tikzcd}\]
        The quiver representations of $Q$ over $K$ form the category $\Rep(Q,K)$. If $K$  is not $\bF_1$ then we denote the space of homomorphisms by $\Hom_{KQ}(V,W)$.
    \end{Definition}
    As before many concepts from representation theory can be adapted to the $\bF_1$-setting. The category $\Rep(Q,\bF_1)$ satisfies versions of the Krull-Schmidt and Jordan-Hölder theorems, as proven in \cite[section 4]{SZCZESNY}. The goal of this paper is to study relations between $\bF_1$-representations and their scalar extensions which will be introduced next.
    \begin{Definition}
        Let $K$ be a field. There is a faithful exact functor \[-\otimes_{\bF_1}K\colon \Vect(\bF_1)\rightarrow\Vect(K)\] called \textit{scalar extension}. It sends an $\bF_1$-vector space $V$ to the free $K$-vector space on $V\backslash\{0_V\}$. We obtain a faithful exact functor \[-\otimes_{\bF_1}K\colon \Rep(Q,\bF_1)\rightarrow\Rep(Q,K)\] by applying the previous functor pointwise. To simplify notation we will write $V^K$ for $V\otimes_{\bF_1}K$. These functors are never full and rarely dense.
    \end{Definition}
\section{Coefficient quivers}
    Coefficient quivers for quiver representations over $\bF_1$ were introduced by Jun and Sistko in \cite{JS1}. They reinterpret an $\bF_1$-representation $V$ of $Q$ as a quiver $\Gamma_V$ with a special map $c\colon \Gamma_V\rightarrow Q$ called a winding. This provides a powerful connection to graph theory.
    \begin{Definition}
        Let $\Gamma$ be a quiver and $c\colon \Gamma\rightarrow Q$ a morphism of quivers. The map $c$ is called a \textit{winding} if for all $\alpha\in Q_1$ and all $a,b\in c^{-1}(\alpha)$ with $a\neq b$ we have $s(a)\neq s(b)$ and $t(a)\neq t(b)$. This condition can be visualised as follows: Consider a subquiver of $\Gamma$ of one of the following forms.
        \[\begin{tikzcd}
        	\bullet & \bullet & \bullet && \bullet & \bullet & \bullet
        	\arrow[from=1-1, to=1-2]
        	\arrow[from=1-3, to=1-2]
        	\arrow[from=1-6, to=1-5]
        	\arrow[from=1-6, to=1-7]
        \end{tikzcd}\]
        Then the two arrows in the subquiver must be mapped to different arrows in $Q_1$.
    \end{Definition}
    It is sometimes useful to interpret a winding $c\colon \Gamma \rightarrow Q$ as a colouring of $\Gamma$. An arrow $b$ with $c(b) = \alpha$ would be considered an $\alpha$-coloured arrow and a vertex $w$ with $c(w) = i$ would be considered an $i$-coloured vertex.
    \begin{Definition}
        Let $V$ be an $\bF_1$-representation of $Q$. The \textit{coefficient quiver} of $V$ is the pair $(\Gamma_V,c_V)$ where $\Gamma_V$ is the following quiver:
        \begin{align*}
            (\Gamma_V)_0 &=\coprod_{i\in Q_0} V_i\backslash\{0_{V_i}\}\\
            (\Gamma_V)_1 &=\coprod_{\alpha\in Q_1} \left\{(v,w)\in ((\Gamma_V)_0)^2\mid V_\alpha(v) = w \right\}\\
            s((v,w)) &= v\\
            t((v,w)) &= w
        \end{align*}
        and $c_V$ is induced by the disjoint unions above.
    \end{Definition}
    The injectivity condition of $\bF_1$-linear maps ensures that $c_V$ is a winding.
    \begin{Example}
        We consider the two loop quiver
        \[\bL_2 = \begin{tikzcd}
        	*
         \arrow["a",color={rgb,255:red,92;green,92;blue,214},loop, in=160, out=200 ,distance=3em, from=1-1, to=1-1]
        \arrow["b",color={rgb,255:red,214;green,92;blue,92},loop, in=-20, out=20 ,distance=3em, from=1-1, to=1-1]
        \end{tikzcd}\]
        and the $\bF_1$-representation $V=(V_*,V_a,V_b)$ given by $V_* = \{1,2\}\cup \{0\}$ and 
        \begin{align*}
            V_a(1) &= 1& V_b(1) &= 2\\
            V_a(2) &= 0& V_b(2) &= 1.
        \end{align*}
        Then the coefficient quiver of $V$ is given by
        \[\Gamma_V =\begin{tikzcd}
        	1 & 2.
        	\arrow[color={rgb,255:red,214;green,92;blue,92}, curve={height=-6pt}, from=1-1, to=1-2]
        	\arrow[color={rgb,255:red,214;green,92;blue,92}, curve={height=-6pt}, from=1-2, to=1-1]
            \arrow[color={rgb,255:red,92;green,92;blue,214},loop, in=160, out=200 ,distance=3em, from=1-1, to=1-1]
        \end{tikzcd}\]
    \end{Example}
    \begin{Definition}
        Let $c\colon \Gamma \rightarrow Q$ be a winding. A full subquiver $H$ of $\Gamma$ is called
        \begin{enumerate}[label = (\roman*)]
            \item \textit{successor closed} if for every arrow $a\in \Gamma$ we have
            \[s(a)\in H\Rightarrow t(a)\in H.\]
            \item \textit{predecessor closed} if for every arrow $a\in \Gamma$ we have
            \[t(a)\in H\Rightarrow s(a)\in H.\]
        \end{enumerate}
    \end{Definition}
    \begin{Proposition}
        Let $V$ and $W$ be $\bF_1$-representations over $Q$.
        \begin{enumerate}[label = (\roman*)]
            \item If $W$ is a subrepresentation of $V$ then $\Gamma_W$ is a successor closed subquiver of $\Gamma_V$.
            \item If $W$ is a quotient representation of $V$ then $\Gamma_W$ is a predecessor closed subquiver of $\Gamma_V$.
            \item $\Gamma_{V\oplus W} = \Gamma_V \sqcup \Gamma_W$.
            \item The indecomposable direct summands of $V$ correspond to the connected components of $\Gamma_V$.         
        \end{enumerate}
    \end{Proposition}
    One can define a category $\mathcal{C}_Q$ of windings over $Q$ that is equivalent to $\Rep(Q,\bF_1)$. The morphisms are based on the following observation: Let $f\colon V\rightarrow W$ be a morphism of quiver representations over $\bF_1$. Then $f$ factors as
    \[V\rightarrow V/\ker(f)\xrightarrow{\overline{f}} \Image(f)\rightarrow W\]
    where $\overline{f}$ is an isomorphism. The triple $(V/\ker(f),\Image(f),\overline{f})$ can be interpreted in the language of windings.
    \begin{Definition}
        Let $c\colon \Gamma\rightarrow Q$ and $c'\colon \Gamma'\rightarrow Q$ be two windings. A \textit{morphism of windings} $\Phi\colon (\Gamma,c)\rightarrow(\Gamma',c')$ is a triple $\Phi=(F,U,\phi)$ where
        \begin{enumerate}[label = (\roman*)]
            \item $F$ is a predecessor closed subquiver of $\Gamma$,
            \item $U$ is a successor closed subquiver of $\Gamma'$,
            \item $\phi\colon F\rightarrow U$ is an isomorphisms with $c'\circ\phi=c$.
        \end{enumerate}
    Let $(F,U,\phi)\colon \Gamma_1 \rightarrow \Gamma_2$ and $(F',U',\psi)\colon \Gamma_2 \rightarrow \Gamma_3$ be two morphisms of windings. Their composition is given by $(F\cap \phi^{-1}(F'),U'\cap\psi(U),\psi\circ\phi)$ where $\psi\circ\phi$ is considered with the appropriate domain and codomain. We obtain a category $\mathcal{C}_Q$ of windings over $Q$.
    \end{Definition}
    \begin{Theorem}[\cite{JS2}]
        The assignment $V\mapsto \Gamma_V$ defines an equivalence of categories
        \[\Rep(Q,\bF_1)\rightarrow \mathcal{C}_Q.\]
    \end{Theorem}
    This equivalence allows us to freely switch between an $\bF_1$-representation and its coefficient quiver depending on which one is more convenient. The tensor product of two representations will be central to this paper. We study it next.
    \begin{Definition}
        Let $V$ and $W$ be in $\Rep(Q,\bF_1)$. The \textit{tensor product} $V\otimes W$ is given by the vector spaces
        \[(V\otimes W)_i = V_i \otimes W_i = V_i \times W_i/(V_i\times \{0_{W_i}\}\cup \{0_{V_i}\}\times W_i)\]
        and the maps
        \begin{align*}
            (V\otimes W)_\alpha\colon (V\otimes W)_{s(\alpha)}&\rightarrow (V\otimes W)_{t(\alpha)}\\
            v\otimes w &\mapsto V_\alpha(v)\otimes W_\alpha(w).
        \end{align*}
        Here we denote an element $(v,w)\in (V\otimes W)_i$ by $v\otimes w$.
    \end{Definition}
    \begin{Definition}
        Let $c\colon \Gamma \rightarrow Q$ and $c'\colon \Gamma' \rightarrow Q$ be windings. We define the \textit{tensor product of windings} $c\otimes c'\colon \Gamma \otimes \Gamma' \rightarrow Q$ by
        \begin{align*}
            (\Gamma \otimes \Gamma')_0 &= \{(v,w)\in \Gamma_0 \times \Gamma'_0 \mid c(v) = c'(w)\}\\
            (\Gamma \otimes \Gamma')_1 &= \{(a,b)\in \Gamma_1 \times \Gamma'_1 \mid c(a) = c'(b)\}\\
            s((a,b)) &= (s(a),s(b))\\
            t((a,b)) &= (t(a),t(b))\\
            c\otimes c'((v,w)) &= c(v)\\
            c\otimes c'((a,b)) &= c(a).
        \end{align*}
    \end{Definition}
    We will give an example of the tensor product of two windings in Example \ref{cover:example}. Comparing definitions yields:
    \begin{Proposition}
        Let $V,W\in \Rep(Q,\bF_1)$. Then
        \[\Gamma_{V\otimes W} \cong \Gamma_V \otimes \Gamma_W.\]
    \end{Proposition}
    
\section{Morphisms between scalar extensions}
    Let $V,W\in \Rep(Q,\bF_1)$ and let $K$ be a field. We want to use the combinatorial structure of the $\bF_1$-representations to obtain a basis of $\Hom_{KQ}(V^K,W^K)$. Our approach is inspired by \cite{Krause}. By definition we have an inclusion
    \[\Hom_{KQ}(V^K,W^K)\subset \prod_{i\in Q_0}\Hom_K(V^K_i,W^K_i).\]
    The representations $V^K$ and $W^K$ carry distinguished bases given by the non-zero elements in $V$ respectively $W$. For $i\in Q_0$, $v\in V_i\backslash\{0_{V_i}\}$, $w\in W_i\backslash\{0_{W_i}\}$ we define the linear map $\mathbf{b}_{v,w}\colon  V^K_i\rightarrow W^K_i$ by 
        \[\mathbf{b}_{v,w}(v')=\delta_{v,v'}w\]
    for all $v'\in V_i\backslash\{0_{V_i}\}$. Then the set
    \[\mathcal{B}_{V,W}=\left\{\mathbf{b}_{v,w} \in \Hom_K(V^K_i,W^K_i) \mid i\in Q_0, v\in V_i\backslash\{0_{V_i}\}, w\in W_i\backslash\{0_{W_i}\}\right\}\]
    forms a basis of $\prod_{i\in Q_0}\Hom_K(V^K_i,W^K_i)$. For an element $f\in \Hom_{KQ}(V^K,W^K)$ we define its coefficients $f_{v,w}$ by
    \[f = \sum_{\mathbf{b}_{v,w}\in \mathcal{B}_{V,W}} f_{v,w}\mathbf{b}_{v,w}.\]
    To simplify notation we set $f_{0,w} = f_{v,0} = 0$.
    \begin{Definition}
        We define an equivalence relation $\sim$ on $\mathcal{B}_{V,W}$ by $\mathbf{b}_{v,w} \sim\mathbf{b}_{v',w'}$ if and only if for all $f\in \Hom_{KQ}(V^K,W^K)$ we have $f_{v,w} = f_{v',w'}$.
    \end{Definition}
    There is at most one equivalence class consisting of all basis elements whose coefficient is always zero. We call it the zero class. We will show that the non-zero equivalence classes define a basis of $\Hom_{KQ}(V^K,W^K)$. To that end we will derive a combinatorial description of the equivalence relation using the tensor product in $\Rep(Q,\bF_1)$.  The connection is given by the bijection \begin{align*}
        (\Gamma_{V\otimes W})_0 &\rightarrow \mathcal{B}_{V,W}\\
        (v,w)&\mapsto \mathbf{b}_{v,w}.
    \end{align*}
    \begin{Example}\label{Hom:exampleone}
        We consider the two loop quiver
        \[\bL_2 = \begin{tikzcd}
        	*
         \arrow["a",color={rgb,255:red,92;green,92;blue,214},loop, in=160, out=200 ,distance=3em, from=1-1, to=1-1]
        \arrow["b",color={rgb,255:red,214;green,92;blue,92},loop, in=-20, out=20 ,distance=3em, from=1-1, to=1-1]
        \end{tikzcd}\]
        and two $\bF_1$-representations $V$ and $W$ with coefficient quivers
        \begin{align*}
            \Gamma_V &=\begin{tikzcd}[ampersand replacement=\&]
        	1 \& 2 \& 3
        	\arrow[color={rgb,255:red,92;green,92;blue,214}, from=1-1, to=1-2]
        	\arrow[color={rgb,255:red,214;green,92;blue,92}, from=1-2, to=1-3]
        \end{tikzcd},& \Gamma_W &=
        \begin{tikzcd}[ampersand replacement=\&]
        	4 \& 5.
        	\arrow[color={rgb,255:red,92;green,92;blue,214}, from=1-1, to=1-2]
        \end{tikzcd}
        \end{align*}
        The tensor product $\Gamma_{V\otimes W}$ is given by
        \[\begin{tikzcd}
        	{(1,4)} & {(2,4)} & {(3,4)} \\
        	{(1,5)} & {(2,5)} & {(3,5)}
        	\arrow[color={rgb,255:red,92;green,92;blue,214}, from=1-1, to=2-2]
        \end{tikzcd}\]
        The non-zero equivalence classes are $\{\mathbf{b}_{1,5}\}$ and $\{\mathbf{b}_{1,4},\mathbf{b}_{2,5}\}$.
    \end{Example}
    \begin{Lemma}\label{Hom:comprelation}
        Let $(v,w)$ and $(v',w')$ be in the same connected component of $\Gamma_{V \otimes W}$. Then $\mathbf{b}_{v,w} \sim\mathbf{b}_{v',w'}$
    \end{Lemma}
    \begin{proof}
        Without loss of generality we can assume that there is an arrow $(a,b)\in (\Gamma_{V\otimes W})_1$ with 
        \begin{align*}
            s((a,b)) &= (v,w)& t((a,b))&=(v',w').
        \end{align*}
        We set $i=c_V(v)=c_W(w)$, $j=c_V(v')=c_W(w')$ and $\alpha = c_V(a)=c_W(b)$. Let $f\in \Hom_{KQ}(V^K,W^K)$. We consider the commutative square
        \[\begin{tikzcd}
        	{V^K_i} & {W^K_i} \\
        	{V^K_j} & {W^K_j}
        	\arrow["{f_i}", from=1-1, to=1-2]
        	\arrow["{W_\alpha}", from=1-2, to=2-2]
        	\arrow["{V_\alpha}"', from=1-1, to=2-1]
        	\arrow["{f_j}"', from=2-1, to=2-2]
        \end{tikzcd}\]
        and compute
        \[W_\alpha(f_i(v)) = W_\alpha(\sum_{\overline{w}\in W_i}f_{v,\overline{w}}\cdot\overline{w})=\sum_{\overline{w}\in W_i}f_{v,\overline{w}}\cdot W_\alpha(\overline{w})\]
        \[f_j(V_\alpha(v))=\sum_{\Tilde{w}\in W_j}f_{V_\alpha(v),\Tilde{w}} \cdot \Tilde{w}=\sum_{\Tilde{w}\in W_j}f_{v',\Tilde{w}} \cdot \Tilde{w}.\]
        We want to examine the coefficient of $w'$. The injectivity condition for $\bF_1$-linear maps implies that $W_\alpha(\overline{w}) = w'$ if and only if $\overline{w} = w$. Hence the coefficient of $w'$ in  $W_\alpha(f_i(v))$ is $f_{v,w}$. The coefficient of $w'$ in $f_j(V_\alpha(v))$ is $f_{v',w'}$. The Lemma follows from commutativity of the square.
    \end{proof}
    Next we will characterise the connected components of $\Gamma_{V\otimes W}$ that correspond to non-zero equivalence classes.
    \begin{Definition}
        Let $C\subset \Gamma_{V\otimes W}$ be a connected component. We say that $C$
        \begin{enumerate}[label = (\roman*)]
            \item \textit{reflects successors} if for all vertices $(v,w)\in C_0$ and arrows $b\in (\Gamma_W)_1$ with $s(b)=w$, there is an arrow $a\in (\Gamma_V)_1$ with $s(a) = v$ and $c_V(a) = c_W(b)$. In this case the arrow $(a,b)\in C_1$ satisfies $s(a,b) = (v,w)$.
            \[\begin{tikzcd}
            	v & w \\
            	{t(a)} & {t(b).}
            	\arrow[from=1-1, to=1-2]
            	\arrow["b", from=1-2, to=2-2]
            	\arrow[dashed, from=2-1, to=2-2]
            	\arrow["a"', dashed, from=1-1, to=2-1]
            \end{tikzcd}\]
            \item \textit{induces predecessors} if for all vertices $(v',w')\in C_0$ and arrows $a\in (\Gamma_V)_1$ with $t(a)=v'$ there is an arrow $b\in (\Gamma_W)_1$ with $t(b) = w'$ and $c_V(a) = c_W(b)$. In this case the arrow $(a,b)\in C_1$ satisfies $t(a,b) = (v',w')$.
            \[\begin{tikzcd}
            	s(a) & s(b) \\
            	{v'} & {w'.}
            	\arrow[dashed, from=1-1, to=1-2]
            	\arrow["b", dashed, from=1-2, to=2-2]
            	\arrow[from=2-1, to=2-2]
            	\arrow["a"', from=1-1, to=2-1]
            \end{tikzcd}\]
            \item is an \textit{admissible component} if $C$ reflects successors and induces predecessors.
        \end{enumerate}
    \end{Definition}
    \begin{Example}
        In Example \ref{Hom:exampleone} testing each instance of the definition shows that the connected components $\{(1,5)\}$ and $\{(1,4),(2,5)\}$ are admissible, the connected component $\{(3,5)\}$ only reflects successors and the other two connected components satisfy neither condition.
    \end{Example}
    \begin{Example}\label{Hom:TrueHoms}
        Let $f\colon V\rightarrow W$ be a morphism of quiver representations over $\bF_1$ and $(F,U,\phi)\colon \Gamma_V\rightarrow \Gamma_W$ the corresponding morphism of windings. We assume that $\Image(f)$ is indecomposable. Then the set 
        \[\{(v,\phi(v))\mid v\in F_0\}\]
        is the vertex set of an admissible component which we will call $C_f$. In this case the projections yield isomorphisms $F\cong C_f\cong U$.
    \end{Example}
    \begin{proof}
        The condition on $\Image(f)$ ensures that $C_f$ is connected. To show that $C_f$ reflects successors let $(v,w)\in (C_f)_0$ be a vertex and $b\in(\Gamma_W)_1$ an arrow with $s(b) = w$. Then $b$ is in $U$ because $U$ is successor closed. Now the arrow $a = \phi^{-1}(b)$ satisfies $s(a) = \phi^{-1}(s(b)) = v$ and $c_V(a) = c_W(b)$ as desired. That $C_f$ induces predecessors follows similarly from $F$ being predecessor closed.
    \end{proof}
    An example of an admissible component that is not of this form will be given in Example \ref{cover:example}.
    \begin{Lemma}
        Let $C\subset \Gamma_{V\otimes W}$ be an admissible component. We consider the element of $\prod_{i\in Q_0}\Hom_K(V^K_i,W^K_i)$ given by 
        \[\mathbf{b}^C=\sum_{(v,w)\in C_0} \mathbf{b}_{v,w}.\]
        Then $\mathbf{b}^C\in \Hom_{KQ}(V^K,W^K)$.
    \end{Lemma}
    \begin{proof}
        We verify this on basis vectors. Let $v\in V_i\backslash \{0_{V_i}\}$ and let $\alpha\in Q_1$ with $s(\alpha)=i$. We once more consider the diagram
        \[\begin{tikzcd}
        	{V^K_i} & {W^K_i} \\
        	{V^K_j} & {W^K_j}
        	\arrow["{\mathbf{b}^C_i}", from=1-1, to=1-2]
        	\arrow["{W_\alpha}", from=1-2, to=2-2]
        	\arrow["{V_\alpha}"', from=1-1, to=2-1]
        	\arrow["{\mathbf{b}^C_j}"', from=2-1, to=2-2]
        \end{tikzcd}\]
        and compute
        \[W_\alpha(\mathbf{b}^C_i(v))=W_\alpha(\sum_{(v,\overline{w})\in C_0}\overline{w}) = \sum_{(v,\overline{w})\in C_0} W_\alpha(\overline{w})\]
        \[\mathbf{b}^C_j(V_\alpha(v)) = \sum_{(V_\alpha(v),\Tilde{w})\in C_0} \Tilde{w}.\]
        The former sum has no repeated non-zero summands due to $\bF_1$-linearity. Thus we only have to show that the following two sets are equal:
        \begin{align*}
            &A :=\{W_\alpha(\overline{w})\mid (v,\overline{w})\in C_0\}\backslash\{0_{W_j}\}& &\textnormal{and}& &B:=\{\Tilde{w}\mid (V_\alpha(v),\Tilde{w})\in C_0\}.
        \end{align*}
        Let $W_\alpha(\overline{w})$ be in $A$. Then there exists an arrow $b\in (\Gamma_W)_1$ with $s(b) = \overline{w}$ and $c_W(b) = \alpha$. As $C$ reflects successors there exists an arrow $a$ in $(\Gamma_V)_1$ with $s(a) = v$ and $c_V(a) = c_W(b) = \alpha$. Together they form an arrow $(a,b)\in C_1$ satisfying $t((a,b)) = (V_\alpha(v),W_\alpha(\overline{w}))\in C_0$. This proves $W_\alpha(\overline{w})\in B$ and hence $A\subset B$. Similarly $B\subset A$ is a consequence of $C$ inducing predecessors.
    \end{proof}
    \begin{Example}
        Let $f\colon V\rightarrow W$ be a morphism of quiver representations over $\bF_1$ and let $C_f$ be the connected component constructed from $f$ in Example \ref{Hom:TrueHoms}. Then $\mathbf{b}^{C_f} = f^K$.
    \end{Example}
    \begin{Lemma}\label{Hom:zero}
        Let $C\subset \Gamma_{V\otimes W}$ be a connected component that is not admissible. Then $C$ is part of the zero class.
    \end{Lemma}
    \begin{proof}
        By Lemma \ref{Hom:comprelation} it suffices to prove this for one element of $C_0$. We will only consider the case where $C_0$ does not reflect successors. A similar argument works when $C_0$ does not induce predecessors. Since $C$ does not reflect successors there is $(v,w)\in C_0$ and $b\in (\Gamma_W)_1$ such that $s(b)=w$ but there is no arrow $a\in (\Gamma_V)_1$ with $s(a) = v$ and $c_V(a) = c_W(b)$. Let $f\in \Hom_{KQ}(V^K,W^K)$. We set $\alpha = c_W(b)$, $i=s(\alpha)$ and $j=t(\alpha)$. From the proof of Lemma \ref{Hom:comprelation} we know
        \[\sum_{\overline{w}\in W_i}f_{v,\overline{w}}\cdot W_\alpha(\overline{w}) = \sum_{\Tilde{w}\in W_j}f_{V_\alpha(v),\Tilde{w}} \cdot \Tilde{w}.\]
        The right hand side vanishes because $V_\alpha(v) = 0$. Since $W_\alpha(w) = t(b) \neq 0$, we must have $f_{v,w} = 0$.        
    \end{proof}
    \begin{Corollary}\label{Hom:bijection}
        The map 
        \begin{align*}
            \phi\colon \{\textnormal{admissible components of }\Gamma_{V\otimes W}\}&\rightarrow \{\textnormal{non-zero equivalence classes in } \mathcal{B}_{V,W}\}\\
            C&\mapsto \{\mathbf{b}_{v,w} \mid (v,w)\in C_0\}
        \end{align*}
        is a bijection.
    \end{Corollary}
    \begin{proof}
        We check that it is well-defined first. Let $C$ be an admissible component. Then $\phi(C)$ is a subset of an equivalence class by Lemma \ref{Hom:comprelation}. The map $\mathbf{b}^C$ witnesses that this equivalence class is not the zero class and that it is contained in $\phi(C)$. This shows that $\phi$ is well-defined. It is injective by construction. Let $\mathbf{b}_{v,w}$ be part of a non-zero class. Then the connected component of $(v,w)$ must be admissible by Lemma \ref{Hom:zero}. This connected component is a preimage of the class of $\mathbf{b}_{v,w}$ under $\phi$.
    \end{proof}
    \begin{Theorem}\label{Hom:basis}
        Let $V$ and $W$ be $\bF_1$-representations of $Q$ and let $K$ be a field. Then the set 
        \[\mathscr{B}_{V,W} = \{\mathbf{b}^C \mid C \textnormal{ admissible component of } \Gamma_{V\otimes W}\}\]
        is a basis of $\Hom_{KQ}(V^K,W^K)$.
    \end{Theorem}
    \begin{proof}
        The $\mathbf{b}^C$ are linearly independent because they are sums over disjoint sets of basis vectors in $\mathcal{B}_{V,W}$. Let $f\in \Hom_{KQ}(V^K,W^K)$. By definition we have
        \[f = \sum_{\mathbf{b}_{v,w}\in \mathcal{B}_{V,W}} f_{v,w}\mathbf{b}_{v,w}.\]
        We drop all summands from the zero class. Then we group the remaining summands according to their equivalence classes and apply Corollary \ref{Hom:bijection}. This yields
        \[f = \sum_C\sum_{(v,w)\in C_0} f_{v,w}\mathbf{b}_{v,w} = \sum_C f_C \mathbf{b}^C\]
        where the sums are over all admissible components and $f_C = f_{v,w}$ for any $(v,w)\in C_0$. Thus $\mathscr{B}_{V,W}$ is a generating set and hence a basis.
    \end{proof}
    \begin{Remark}
        The special case where $\Gamma_V$ and $\Gamma_W$ are trees or aperiodic cycles was studied by Crawley-Boevey \cite{CB} and Krause \cite{Krause}. The latter used so-called \textit{admissible triples} to construct a basis of $\Hom_{KQ}(V^K,W^K)$. There is a bijection
        \begin{align*}
            \{\textnormal{admissible components of }\Gamma_{V\otimes W}\}&\rightarrow\{\textnormal{admissible triples connecting }\Gamma_V \textnormal{ and }\Gamma_W\} \\
            C&\mapsto (C,\pi_1,\pi_2)
        \end{align*}
        where the maps $\pi_1$ and $\pi_2$ are given by
        \begin{align*}
            \pi_1\colon C &\rightarrow \Gamma_V& \pi_2\colon C &\rightarrow \Gamma_W\\
            (v,w)&\mapsto v& (v,w)&\mapsto w.
        \end{align*}
        In this way we recover the standard basis for morphisms between tree modules.
    \end{Remark}
\section{Covering components}
    As an application we want to study when a scalar extension $V^K$ of an indecomposable $\mathbb{F}_1$-representation $V$ remains indecomposable. We will relate this to the presence or absence of special admissible components of $\Gamma_{V\otimes V}$ which we will call covering components. They encode an unusual kind of symmetry of $V$.
    \begin{Definition}
        Let $Q$ and $Q'$ be quivers. A morphism of quivers $\pi\colon Q\rightarrow Q'$ is a \textit{covering map} if for every vertex $v\in Q'$, vertex $v'\in \pi^{-1}(v)$ and arrow $a\in Q_0$ with $s(a) = v$ (or $t(a) = v$) there is a unique arrow $a'\in \pi^{-1}(a)$ with $s(a') = v'$ (respectively $t(a') = v'$).
    \end{Definition}
    \begin{Definition}
        Let $V,W\in \Rep(Q,\bF_1)$. A connected component $C$ of $\Gamma_{V\otimes W}$ is called a \textit{covering component} if the two projections
        \begin{align*}
            \pi_1\colon C &\rightarrow \Gamma_V& \pi_2\colon C &\rightarrow \Gamma_W\\
            (v,w)&\mapsto v& (v,w)&\mapsto w
        \end{align*}
        are covering maps.
    \end{Definition}
    \begin{Example}\label{cover:automorph}
        Let $V$ be indecomposable. Then the admissible component $C_\sigma$ induced by an automorphism $\sigma\colon V\rightarrow V$ is a covering component. In this case both projections are isomorphisms.
    \end{Example}
    \begin{Example}\label{cover:example}
        We consider the following representation over $\mathbb{L}_2$:
        \[\begin{tikzcd}
        	1 & 2 & 3
         \arrow[color={rgb,255:red,92;green,92;blue,214},loop, in=160, out=200 ,distance=3em, from=1-1, to=1-1]
             \arrow[color={rgb,255:red,214;green,92;blue,92},loop, in=-20, out=20 ,distance=3em, from=1-3, to=1-3]
        	\arrow[color={rgb,255:red,214;green,92;blue,92}, curve={height=-6pt}, from=1-1, to=1-2]
        	\arrow[color={rgb,255:red,92;green,92;blue,214}, curve={height=-6pt}, from=1-2, to=1-3]
        	\arrow[color={rgb,255:red,92;green,92;blue,214}, curve={height=-6pt}, from=1-3, to=1-2]
        	\arrow[color={rgb,255:red,214;green,92;blue,92}, curve={height=-6pt}, from=1-2, to=1-1]
        \end{tikzcd}\]
        Here the colours blue and red correspond to the two arrows of $\mathbb{L}_2$. Its tensor square is given by
        
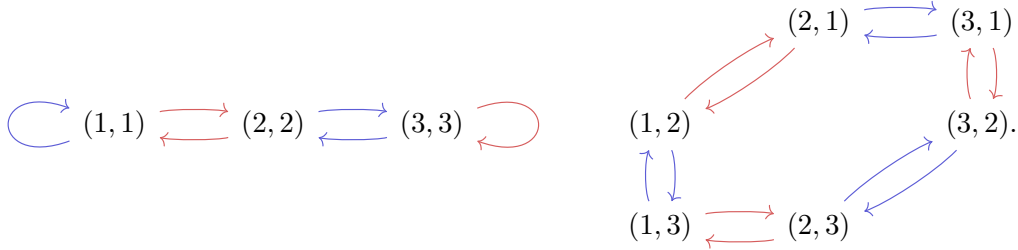
\[\begin{tikzcd}
        	&&&&& {(2,1)} & {(3,1)} \\
        	{(1,1)} & {(2,2)} & {(3,3)} && {(1,2)} && {(3,2).} \\
        	&&&& {(1,3)} & {(2,3)}
             \arrow[color={rgb,255:red,92;green,92;blue,214},loop, in=160, out=200 ,distance=3em, from=2-1, to=2-1]
             \arrow[color={rgb,255:red,214;green,92;blue,92},loop, in=-20, out=20 ,distance=3em, from=2-3, to=2-3]
        	\arrow[color={rgb,255:red,214;green,92;blue,92}, curve={height=-6pt}, from=1-6, to=2-5]
        	\arrow[color={rgb,255:red,214;green,92;blue,92}, curve={height=-6pt}, from=2-5, to=1-6]
        	\arrow[color={rgb,255:red,92;green,92;blue,214}, curve={height=-6pt}, from=3-6, to=2-7]
        	\arrow[color={rgb,255:red,92;green,92;blue,214}, curve={height=-6pt}, from=2-7, to=3-6]
        	\arrow[color={rgb,255:red,214;green,92;blue,92}, curve={height=-6pt}, from=3-6, to=3-5]
        	\arrow[color={rgb,255:red,214;green,92;blue,92}, curve={height=-6pt}, from=3-5, to=3-6]
        	\arrow[color={rgb,255:red,92;green,92;blue,214}, curve={height=-6pt}, from=3-5, to=2-5]
        	\arrow[color={rgb,255:red,92;green,92;blue,214}, curve={height=-6pt}, from=2-5, to=3-5]
        	\arrow[color={rgb,255:red,92;green,92;blue,214}, curve={height=-6pt}, from=1-6, to=1-7]
        	\arrow[color={rgb,255:red,92;green,92;blue,214}, curve={height=-6pt}, from=1-7, to=1-6]
        	\arrow[color={rgb,255:red,214;green,92;blue,92}, curve={height=-6pt}, from=1-7, to=2-7]
        	\arrow[color={rgb,255:red,214;green,92;blue,92}, curve={height=-6pt}, from=2-7, to=1-7]
        	\arrow[color={rgb,255:red,214;green,92;blue,92}, curve={height=-6pt}, from=2-1, to=2-2]
        	\arrow[color={rgb,255:red,92;green,92;blue,214}, curve={height=-6pt}, from=2-2, to=2-3]
        	\arrow[color={rgb,255:red,214;green,92;blue,92}, curve={height=-6pt}, from=2-2, to=2-1]
        	\arrow[color={rgb,255:red,92;green,92;blue,214}, curve={height=-6pt}, from=2-3, to=2-2]
        \end{tikzcd}\]
        Both connected components are covering components. The left one is induced by the identity. This is the smallest representation with two covering components but no non-trivial automorphism.
    \end{Example}
    The next proposition will give alternative characterisations of covering components. Let \begin{align*}
        \tau\colon V\otimes W &\xrightarrow{\cong} W\otimes V\\
        v\otimes w &\mapsto w\otimes v
    \end{align*}
    be the standard braiding.
    \begin{Proposition}\label{cover:six}
        Let $C$ be a connected component of $\Gamma_{V\otimes W}$. Then the following are equivalent:
        \begin{enumerate}[label = (\roman*)]
            \item $C$ is a covering component.
            \item $C$ and $\tau(C)$ are admissible components.
            \item Let $(v,w)\in C_0$ and $\alpha\in Q_1$ be an arrow with $s(\alpha) = c_V(v)$. Then the following are equivalent:
            \begin{enumerate}[label = (\arabic*)]
                \item There is an arrow $a\in (\Gamma_V)_1$ with $s(a) = v$ and $c_V(a) = \alpha$.
                \item There is an arrow $(a,b)\in C_1$ with $s((a,b))=(v,w)$ and $c_{V\otimes W}((a,b)) = \alpha$.
                \item There is an arrow $b\in (\Gamma_W)_1$ with $s(b) = w$ and $c_W(b) = \alpha$.
            \end{enumerate}
            Dually let $\beta \in Q_1$ with $t(\beta) = c_V(v)$. Then the following are equivalent:
            \begin{enumerate}[label = (\arabic*')]
                \item There is an arrow $a\in (\Gamma_V)_1$ with $t(a) = v$ and $c_V(a) = \beta$.
                \item There is an arrow $(a,b)\in C_1$ with $t((a,b))=(v,w)$ and $c_{V\otimes W}((a,b)) = \beta$.
                \item There is an arrow $b\in (\Gamma_W)_1$ with $t(b) = w$ and $c_W(b) = \beta$.
            \end{enumerate}
        \end{enumerate}
    \end{Proposition}
    \begin{proof}
        We first note that by the construction of the tensor product we have
        \begin{align}\label{cover:trivialImplications}
            [(1)\text{ and }(3)] &\Leftrightarrow (2)& &\text{and}& [(1')\text{ and }(3')] &\Leftrightarrow (2').
        \end{align}
        Thus the only implications in $(iii)$ that can fail are $(1)\Rightarrow (2) \Leftarrow (3)$ and $(1')\Rightarrow (2') \Leftarrow (3')$.
        $(i)\Leftrightarrow (iii)$: Let $v\in (\Gamma_V)_0$, $(v,w)\in \pi_1^{-1}(v)$ and $a\in (\Gamma_V)_1$ with $s(a) = v$. For $\pi_1$ to be a covering map there needs to be unique arrow $(a,b)\in C_1$ with $s(a,b) = (v,w)$. Uniqueness is guaranteed by $\pi_1$ being a winding while existence is exactly the condition $(1)\Rightarrow (2)$. Repeating the argument for arrows leaving $V$ shows that $\pi_1$ is a covering map if and only if the implications $(1)\Rightarrow(2)$ and $(1')\Rightarrow(2')$ hold. Similarly $\pi_2$ is a covering map if and only if the implications $(2)\Leftarrow(3)$ and $(2')\Leftarrow(3')$ hold.\\
        $(ii)\Leftrightarrow (iii)$: The definition of '$C$ reflects successors' is exactly the implication $(3)\Rightarrow(1)$ and the definition of '$C$ induces predecessors' is $(1')\Rightarrow (3')$ is the condition that . Similarly the implications $(3')\Rightarrow (1')$ and $(1)\Rightarrow (3)$ are equivalent to $\tau(C)$ being an admissible component. Altogether $(ii)$ is equivalent to $(1)\Leftrightarrow (3)$ and $(1')\Leftrightarrow (3')$ which in view of \eqref{cover:trivialImplications} is equivalent to $(iii)$.
    \end{proof}
    We now restrict our focus to endomorphisms of an indecomposable representation. In this case there is a much simpler criterion.
    \begin{Proposition}\label{cover:criterium}
        Let $V$ be indecomposable and $C$ a connected component of $\Gamma_{V\otimes V}$. Then the following are equivalent:
        \begin{enumerate}[label = (\roman*)]
            \item $C$ is a covering component,
            \item $C$ is admissible and $\pi_1\colon C\rightarrow \Gamma_V$ is surjective on vertices.
        \end{enumerate}
    \end{Proposition}
    \begin{proof}
        $(i)\Rightarrow(ii)$: The coefficient quiver $\Gamma_V$ is connected since $V$ is indecomposable. Any non-empty covering of a connected quiver  is surjective.\\
        $(ii)\Rightarrow(i)$: By Proposition \ref{cover:six} we need to show that $\tau(C)$ is an admissible component. This is equivalent to the implications $(2')\Leftarrow (3')$ and $(1)\Rightarrow (2)$ from the same proposition. We will only prove $(1)\Rightarrow (2)$. The other implication follows from a similar argument.
        \begin{Claim}
            Let $(v,w)\in C_0$ and $a\in (\Gamma_V)_1$ with $s(a) = v$. Assume that there is a sequence $(v_0,\dots,v_n)$ of vertices of $\Gamma_V$ with $v_0=v_n=v$, $v_1=w$ and $(v_{i-1},v_i)\in C_0$ for $1\leq i\leq n$. Then there is an arrow $b\in(\Gamma_V)_1$ with $s(b) = w$ and $c_V(a)=c_v(b)$.
        \end{Claim}
        \begin{proof}
            We consider the vertex $(v_{n-1},v_n)$. The arrow $a$ starts at $v_n=v$ and $C$ reflects successors, hence there must be an arrow $a_{n-1}\in(\Gamma_V)_1$ with $s(a_{n-1}) = v_{n-1}$ and $c_V(a_{n-1}) = c_V(a)$. By descending induction there is an arrow $a_1\in(\Gamma_V)_1$ with $s(a_1)=v$ and $c_V(a_1) = c_V(a)$. The desired arrow is given by $b:=a_1$ since $v_1= w$.
            \[\begin{tikzcd}
            	v & w & {v_2} & {...} & {v_{n-1}} & v \\
            	{t(a)} & {t(b)} & {t(a_2)} & {...} & {t(a_{n-1})} & {t(a)}
            	\arrow[from=1-1, to=1-2]
            	\arrow[from=1-2, to=1-3]
            	\arrow[from=1-3, to=1-4]
            	\arrow[from=1-4, to=1-5]
            	\arrow[from=1-5, to=1-6]
            	\arrow["a"', from=1-1, to=2-1]
            	\arrow["b"', dashed, from=1-2, to=2-2]
            	\arrow["a"', from=1-6, to=2-6]
            	\arrow[dashed, from=2-5, to=2-6]
            	\arrow["{a_{n-1}}"', dashed, from=1-5, to=2-5]
            	\arrow[dashed, from=2-4, to=2-5]
            	\arrow[dashed, from=2-3, to=2-4]
            	\arrow["{a_2}"', dashed, from=1-3, to=2-3]
            	\arrow[dashed, from=2-1, to=2-2]
            	\arrow[dashed, from=2-2, to=2-3]
            \end{tikzcd}\]
        \end{proof}
        If we can find such a sequence for each vertex of $C$ then this proves $(1)\Rightarrow (2)$. We start by finding just one vertex that admits such a sequence. Let $v_1\in(\Gamma_V)_0$ be any vertex. Because $\pi_1$ is surjective on vertices there must be another vertex $v_2\in(\Gamma_V)_0$ with $(v_1,v_2)\in C_0$. Using induction we obtain an infinite sequence $(v_i)_{i\in \bN}$ of vertices of $\Gamma_V$ with $(v_i,v_{i+1})\in C_0$ for all $i\in \bN$. By the pigeonhole principle there must be some $i<j\in \bN$ with $v_i = v_j$. The subsequence $(v_i,v_{i+1},\dots,v_j)$ is of the desired form for the vertex $(v_i,v_{i+1})$.
        
        Now let $(v,w)$ and $(v',w')$ be adjacent vertices in $C$ such that $(v,w)$ admits such a sequence. If we can show that $(v',w')$ also admits such a sequence then every vertex admits one and we are done. Let $(a,b)\in(\Gamma_V)_1$ be the arrow connecting $(v,w)$ and $(v',w')$. We assume $s((a,b)) = (v,w)$ and $t((a,b))=(v',w')$. A similar argument works should $(a,b)$ have the opposite orientation. Now let $(v_0,\dots,v_n)$ be a sequence for $(v,w)$. Let $\alpha = c_V(a)$. Because $C$ reflects successors and $v=v_n$ there is an arrow $a_{n-1}\in (\Gamma_V)_1$ with $s(a_{n-1})=v_{n-1}$ and $c_V(a_{n-1}) = \alpha$. By induction there is an arrow $a_i$ with $s(a_i)=v_i$ for all $0\leq i\leq n$. The winding condition gives $a_0=a_n=a$ and $a_1=b$. We obtain a sequence $((a_0,a_1),(a_1,a_2),\dots,(a_{n-1},a_n))$ of arrows in $C_1$. Now the sequence $(t(a_0),t(a_1)\dots,t(a_n))$ satisfies $t(a_0) =t(a)=v'$, $t(a_1)=t(b)=w'$, $t(a_n) = t(a)=v'$ and $(t(a_{i-1}),t(a_i)) = t((a_{i-1},a_i))\in C_0$ for all $1\leq i \leq n$. Hence it is the desired sequence for $(v',w')$ completing the proof.
    \end{proof}
    Finally we give a second criterion for detecting covering components.
    \begin{Proposition}\label{cover:detection}
        Let $V$ be indecomposable and $(C^1,\dots,C^n)$ a sequence of admissible components of $\Gamma_{V\otimes V}$. Further assume that there is a sequence $(v_0,\dots,v_n)$ of vertices of $\Gamma_V$ with $v_0 = v_n$ and $(v_{i-1},v_i)\in C_0^i$ for $1\leq i \leq n$. Then $C^1$ is a covering component. 
    \end{Proposition}
    \begin{proof}
        We consider the endomorphism $\mathbf{b}^{C^n}\cdot\dotso\cdot\mathbf{b}^{C^1}$ of $V^\bC$ and its basis expansion
        \[\mathbf{b}^{C^n}\cdot\dotso\cdot\mathbf{b}^{C^1} = \sum_{\mathbf{b}^C\in \mathscr{B}_{V,V}} \lambda_C \mathbf{b}^C.\]
        Its value at $v_0$ is given by
        \begin{equation}
            \mathbf{b}^{C^n}\cdot\dotso\cdot\mathbf{b}^{C^1}(v_0) = \sum_{(v_{n-1},v_n)\in C^n_0}\dots\sum_{(v_1,v_2)\in C^2_0}\sum_{(v_0,v_1)\in C^1_0} v_n.
        \end{equation}
        the sequence $(v_0,\dots,v_n)$ witnesses that the coefficient of $v_n$ in $\mathbf{b}^{C^n}\cdot\dotso\cdot\mathbf{b}^{C^1}(v_0)$ is not zero. Because $v_0 = v_n$ we get $\lambda_{C_\identity}\neq0$. In particular none of the standard basis vectors of $V^\bC$ lie in the kernel of $\mathbf{b}^{C^n}\cdot\dotso\cdot\mathbf{b}^{C^1}(v_i)$. Then none of them lie in the kernel of $\mathbf{b}^{C^1}$ either. Finally this implies that $\pi_1\colon C^1\rightarrow V$ is surjective and hence that $C^{i+1}$ is a covering component by Proposition \ref{cover:criterium}.
    \end{proof}
    
\section{Absolutely indecomposable representations}
    \begin{Definition}
        An indecomposable $\bF_1$-representation $V$ is called \textit{absolutely indecomposable} if $V^K$ is indecomposable for all fields $K$.
    \end{Definition}
    In this section we will apply our previous considerations to obtain the following combinatorial characterisation of absolutely indecomposable representations.
    \begin{Theorem}\label{abs:theorem}
        Let $V$ be an indecomposable quiver representation over $\bF_1$. Then the following are equivalent:
        \begin{enumerate}[label = (\roman*)]
            \item $V$ is absolutely indecomposable,
            \item $V^\mathbb{C}$ is indecomposable,
            \item $\Gamma_{V\otimes V}$ contains exactly one covering component.
        \end{enumerate}
    \end{Theorem}
    \begin{proof}
        $(i)\Rightarrow(ii)$: This is clear.\\
        $(ii)\Rightarrow(iii)$: This will be Proposition \ref{abs:dec}.\\
        $(iii)\Rightarrow(i)$: This will be Proposition \ref{abs:indec}.
    \end{proof}
    The unique covering component in $(iii)$ is $C_\identity$. We start by showing that any representation with more than one covering component becomes decomposable over $\bC$. To this end we use the following criterion for decomposable modules.
    \begin{Proposition}
        Let $K$ be a field, $A$ a $K$-algebra and $V$ a finite dimensional $A$-module. If there is an endomorphism $f$ of $V$ that has at least two different eigenvalues then $V$ is decomposable.
    \end{Proposition}
    \begin{proof}
        Let $f$ be such an endomorphism and $\lambda \neq \mu$ two of its eigenvalues. By the fitting lemma every endomorphism of an indecomposable module is either nilpotent or an isomorphism. The endomorphism $f-\lambda \identity_V$ is neither.
    \end{proof}
    We also need the following version of the Perron-Frobenius theorem.
    \begin{Theorem}[Perron-Frobenius]
        Let $M$ be the adjacency matrix of a connected, regular directed graph $G$. Then the characteristic polynomial of $M$ has a simple root.
    \end{Theorem}
    See \cite{AGT} for a graph theoretic treatment of the Perron-Frobenius theorem. The theorem is stated there using the weaker assumption that the graph is strongly connected.
    \begin{Proposition}\label{abs:dec}
        Let $V\in \Rep(Q,\bF_1)$ be indecomposable and let $C \neq C_\identity$ be a covering component in $\Gamma_{V\otimes V}$. Then $V^\bC$ is decomposable.
    \end{Proposition}
    \begin{proof}
        We will show that the endomorphism $\mathbf{b}^C$ has multiple eigenvalues by interpreting it as the adjacency matrix of a graph and then applying the Perron-Frobenius theorem. Let $G = (G_0,G_1,s,t)$ be the directed graph given by
        \begin{align*}
            G_0 &= (\Gamma_V)_0\\
            G_1 &= C_0\\
            s((v,w)) &= v\\
            t((v,w)) &= w.
        \end{align*}
        The adjacency matrix of $G$ is equal to the matrix representation of $\mathbf{b}^C$ with respect to the standard basis of $V^\bC$. Let $v\in (\Gamma_V)_0$ be a vertex. A vertex $(v',w')$ of $C_0$ becomes an edge starting at $v$ if and only if $\pi_1((v',w')) = v$. Thus the outdegree of $v$ is the degree of the covering $\pi_1$. This proves that $G$ is $k$-regular for $k = \degree(\pi_1)$. The graph $G$ might not be connected. If it is not then it is a disjoint union of $k$-regular connected graphs and its adjacency matrix decomposes as a block diagonal matrix where each block is the adjacency matrix of a connected component. Let $D$ be such a connected component and let $(v,w)$ be an edge in $D$. Such an edge must exist because $D$ is $k$-regular. If $v=w$ then $C$ intersects $C_{\identity}$. So we must have $v\neq w$ and consequently $D$ consists of multiple vertices. Applying the Perron-Frobenius theorem to the adjacency matrix of $D$ shows that it has multiple different eigenvalues. Then $\mathbf{b}^C$ must have multiple different eigenvalues, too. In particular $V^\bC$ is decomposable.
    \end{proof}
    This completes one direction of our characterisation of absolutely indecomposable representations. For the other direction we need to prove that scalar extensions of representations with just one covering component are indecomposable. We will use the following criterion for indecomposable representations.
    \begin{Proposition}
        Let $K$ be a field, $A$ a $K$-algebra and $V$ a finite dimensional $A$-module. Let $I\subset \End_A(V)$ be a subspace consisting entirely of nilpotent endomorphisms such that $K\cdot\identity_V \oplus I = \End_A(V)$. Then $V$ is indecomposable.
    \end{Proposition}
    \begin{proof}
        Let $f = \lambda\identity + f_{nil}\in \End_A(V)$ with $f_{nil}\in I$. Then the only eigenvalue of $f$ is $\lambda$. Let $V=V_1 \oplus V_2$ be a non-trivial direct sum decomposition. The projection onto $V_1$ has eigenvalues 0 and 1, a contradiction.
    \end{proof}
    \begin{Proposition}\label{abs:indec}
        Let $V\in \Rep(Q,\bF_1)$ be an indecomposable representation with just one covering component. Let $K$ be a field. Then $V^K$ is indecomposable.
    \end{Proposition}
    \begin{proof}
        We will apply the previous proposition to the subspace 
        \[I =\Kspan(\mathscr{B}_{V,V}\backslash \{\mathbf{b}^{C_\identity}\}).\]
        The claim $K\cdot\identity_{V^K} \oplus I = \End_{KQ}(V^K)$ follows from $\mathbf{b}^{C_\identity} = \identity_{V^K} $ and Theorem \ref{Hom:basis}. So we only need to check that $I$ consists of nilpotent morphisms. Let $f\in I$ be non-nilpotent and let $n>|(\Gamma_V)_0|$. The endomorphism $f^n$ is a linear combination of elements of the form
        \[\mathbf{b}^{C^n}\cdot\dotso\cdot\mathbf{b}^{C^1}.\]
        for some sequence $(C^i)$ of admissible components in $\Gamma_{V\otimes V}$ none of which are covering components. We choose a non-zero element of this form. Let $v\in V$ be non-zero. Its image is given by
        \begin{equation}
            \mathbf{b}^{C^n}\cdot\dotso\cdot\mathbf{b}^{C^1}(v) = \sum_{(v_{n-1},v_n)\in C^n_0}\dots\sum_{(v_1,v_2)\in C^2_0}\sum_{(v,v_1)\in C^1_0} v_n.
        \end{equation}
        Because $\mathbf{b}^{C^n}\cdot\dotso\cdot\mathbf{b}^{C^1}$ is not zero there must be a sequence $(v_0,\dots,v_n)$ with $v_i\in V$ not zero and $(v_{i-1},v_i)\in C^i_0$ for $1\leq i \leq n$. By the pigeonhole principle there are $0\leq i<j\leq n$ with $v_i = v_j$. Applying Proposition \ref{cover:detection} to the sequences $(C^{i+1},\dots C^j)$ and $(v_i,\dots v_j)$ shows that $C^{i+1}$ is a covering component, a contradiction.
    \end{proof}
    This criterion yields a practically fast algorithm.
    \begin{Corollary}
        Let $V\in \Rep(Q,\bF_1)$ be indecomposable. Whether $V$ is absolutely indecomposable can be tested in $O(m^2)$ time where $m$ is the number of arrows in $\Gamma_V$.
    \end{Corollary}
    \begin{proof}
        Let $n$ be the number of vertices in $\Gamma_V$. The quiver $\Gamma_V$ is connected because $V$ is indecomposable. In particular $n\leq m+1$. The tensor product $\Gamma_{V\otimes V}$ can be computed in $O(m^2)$ time. We can identify all covering components by finding all admissible components $C$ such that $\tau(C)$ is admissible. Across all connected components there are $nm$ test cases each for reflecting successors, inducing predecessors and their dual properties. Each of those cases can be checked in constant time.
    \end{proof}
    
\section{Representations with finite nice length}
    Cerulli-Irelli \cite{CI} and Haupt \cite{Haupt}  developed a combinatorial technique to compute the Euler characteristic of certain quiver Grassmannians. Jun and Sistko \cite{JS2} adapted these results to the context of $\bF_1$-representations. They introduced the notion of finite nice length as a criterion for when this approach can be applied. We will show that all indecomposable representations with finite nice length are absolutely indecomposable.\\
    Finite nice length representations are defined via a sequence of increasingly fine gradings satisfying certain compatibility conditions.
    \begin{Definition}
        Let $V\in \Rep(Q,\bF_1)$ be a representation. A \textit{grading} of $V$ is a map $\partial\colon (\Gamma_V)_0\rightarrow \bZ$. Let $\underline{\partial}=(\partial_0,\dots,\partial_n)$ be a sequence of gradings for $V$. A grading $\partial_{n+1}$ is $\underline{\partial}$-\textit{nice} if for each pair of arrows $a,b\in (\Gamma_V)_1$ satisfying
        \begin{enumerate}[label = (\roman*)]
            \item $c_V(a) = c_V(b)$,
            \item $\partial_i(s(a)) = \partial_i(s(b))$ for all $0\leq i \leq n$,
            \item $\partial_i(t(a)) = \partial_i(t(b))$ for all $0\leq i \leq n$,
        \end{enumerate}
        we have 
        \[\partial_{n+1}(s(a))-\partial_{n+1}(t(a))=\partial_{n+1}(s(b))-\partial_{n+1}(t(b)).\]
        The sequence $\underline{\partial}$ is \textit{nice} if $\partial_i$ is $(\partial_0,\dots,\partial_{i-1})$-nice for all $0\leq i \leq n$. The representation $V$ has \textit{finite nice length} if there exists a nice sequence $\underline{\partial}=(\partial_0,\dots,\partial_n)$ for $V$ that distinguishes vertices. That is for each pair of vertices $v\neq w\in (\Gamma_V)_0$ there exists $0\leq i \leq n$ with $\partial_i(v)\neq \partial_i(w)$. In this case the \textit{nice length} of $V$ is the smallest $n$ such that there exists a nice sequence $(\partial_0,\dots,\partial_n)$ for $V$ that distinguishes vertices.
    \end{Definition}
    We will prove that indecomposable finite nice length representations are absolutely indecomposable by studying certain strings in the coefficient quiver.
    \begin{Definition}
        Let $\Gamma$ be a quiver. We consider an alphabet consisting of the two symbols $a$ and $a^-$ for each arrow $a\in \Gamma_1$. We define $s(a^-):=t(a)$ and $t(a^-):= s(a)$. A \textit{string} in $\Gamma$ is a non-empty word $a_n\dots a_1$ over this alphabet such that $s(a_{i+1}) = t(a_i)$ for all $1\leq i\leq n-1$. We set $s(a_n\dots a_1) := s(a_1)$ and $t(a_n\dots a_1) := t(a_n)$. If $C = a_n\dots a_1$ and $D = b_m\dots b_1$ are two strings with $t(D)=s(C)$ then we define their \textit{composition} by $CD=a_n\dots a_1b_m\dots b_1$. A string $C$ is called a \textit{cycle} if $s(C) = t(C)$. In this case we denote the $k$-fold composition of $C$ with itself by $C^k$. Let $c\colon \Gamma\rightarrow Q$ be a winding. We extend $c$ to a map on strings by setting $c(a^-) = c(a)^-$ and $c(a_n\dots a_1) = c(a_n)\dots c(a_1)$.
    \end{Definition}
    We will describe a certain type of string within a coefficient quiver that prevents finite nice length. Then we will show that such a string always exists for not absolutely indecomposable representations. 
    \begin{Proposition}\label{finite:criterion}
        Let $V\in \Rep(Q,\bF_1)$ be a representation. If there is a string $E$ in $\Gamma_V$ that is not a cycle but can be completed to a cycle by a string $D$ in such a way that $c(D) = c(E)^s$ for some $s\geq 1$, then $V$ does not have finite nice length.
    \end{Proposition}
    \begin{proof}
        We expand $E = a_m\dots a_1$, $D = a_{m(s+1)}\dots a_{m+1}$. We interpret all indices on these arrows modulo $m(s+1)$. Let $\underline{\partial}=(\partial_0,\dots,\partial_n)$ be a nice sequence for $V$. We will prove by induction that $\partial_i(s(a_k)) = \partial_i(s(a_{k+m}))$ for $0\leq i \leq n$ and $1\leq k \leq m(s+1)$. Then $\underline{\partial}$ does not distinguish vertices since $s(a_1) = s(E)\neq t(E) = s(a_{m+1})$. We assume this holds for all $i<r$. For $r = 0$ this is an empty condition. Our assumptions are chosen such that for any $1\leq k \leq m(s+1)$ the pair of arrows $a_k$ and $a_{k+m}$ satisfies conditions $(i)$-$(iii)$ in the the niceness condition for $\partial_r$. For $1\leq k \leq m(s+1)$ this implies 
        \[\partial_r(s(a_k))-\partial_r(s(a_{k+1}))=\partial_r(s(a_{k+m}))-\partial_r(s(a_{k+m+1})).\]
        A telescope sum along the cycle $DE$ yields
        \[\partial_r(s(a_k))-\partial_r(s(a_{k+m}))=\partial_r(s(a_{k+m}))-\partial_r(s(a_{k+2m})).\]
        We call this difference $\Delta_k$. A second telescope sum shows
        \[0 =\partial_r(s(a_k))-\partial_r(s(a_{k+m(s+1)}))=(s+1)\Delta_k.\]
        In particular, $\Delta_k = 0$ for all $k$ as required. This completes the induction.
    \end{proof}
    \begin{Proposition}
        Let $V\in \Rep(Q,\bF_1)$ be indecomposable and of finite nice length. Then $V$ is absolutely indecomposable.
    \end{Proposition}
    \begin{proof}
        We prove the contraposition. Let $V$ be an indecomposable but not absolutely indecomposable representation. By \ref{abs:theorem} there is a covering component $C\neq C_\identity$ in $\Gamma_{V\otimes V}$. We have to find a string as in the previous proposition. There are two cases to consider depending on the degree of the covering $\pi_1\colon C \rightarrow \Gamma_V$.
        \begin{Claim}
            If $\degree(\pi_1) = 1$ then $V$ does not have finite nice length.
        \end{Claim}
        \begin{proof}
            In this case the projections $\pi_1$ and $\pi_2$ are isomorphisms. Thus $C$ must be of the form $C_\sigma$ for the non-identity automorphism $\sigma =\pi_2\circ(\pi_1)^{-1}\colon V\rightarrow V$. Let $v\in(\Gamma_V)_0$ be a vertex that is not fixed by $\sigma$. Because $\Gamma_V$ is connected there is an undirected path from $v$ to $\sigma(v)$ and hence a string $E$ with $s(E) = v$ and $t(E) = \sigma(v)$. Let $n>1$ be minimal such that $\sigma^n(v) = v$. We set $D = \sigma^{n-1}(E)\dots\sigma(E)$. Then $E$ and $D$ satisfy the conditions of the previous proposition. Hence, $V$ does not have finite nice length.
        \end{proof}
        \begin{Claim}
            If $\degree(\pi_1) \geq 2$ then $V$ does not have finite nice length.
        \end{Claim}
        \begin{innerproof}[Proof]
            We fix an element $v\in (\Gamma_V)_0$ and two different elements $(v,w),(v,w')\in \pi_1^{-1}(v)$. Since $C$ is connected there is an undirected path and hence a string $\Tilde{E}$ from $(v,w)$ to $(v,w')$. We set $E = \pi_2(\Tilde{E})$. In particular $s(E) = w$ and $t(E) = w'$. The path $\pi_1(\Tilde{E})$ is a cycle at $v$. It acts on $\pi_1^{-1}(v)$ because $\pi_1$ is a covering map. Let $n$ be minimal such that $(\pi_1(\Tilde{E}))^n$ lifts to a path from $(v,w')$ to $(v,w)$. We call this path $D$. Then the strings $E$ and $D$ satisfy the conditions of the previous proposition. Hence $V$ does not have finite nice length.\qedhere
        \end{innerproof}
    \end{proof}
    The converse does not hold as the following example shows. 
    \begin{Example}
        An example of an absolutely indecomposable representation with infinite nice length was given by Jun and Sistko in \cite{JS2}. We give a different example here. We consider the following representation of $\bL_2$:
        \[\begin{tikzcd}[ampersand replacement=\&]
        	1 \& 2.
        	\arrow["b"{description}, color={rgb,255:red,214;green,92;blue,92}, curve={height=-9pt}, from=1-1, to=1-2]
        	\arrow["a"{description}, color={rgb,255:red,214;green,92;blue,92}, curve={height=-9pt}, from=1-2, to=1-1]
        	\arrow["c"{description}, color={rgb,255:red,92;green,92;blue,214}, from=1-2, to=1-1]
        \end{tikzcd}\]
        Here the two colours correspond to the two arrows of $\bL_2$. The string $ab$ satisfies the conditions of Proposition \ref{finite:criterion} hence $V$ does not have finite nice length. On the other hand its tensor square is given by
        \[\begin{tikzcd}[ampersand replacement=\&]
        	{(1,1)} \& {(2,2)} \& {(1,2)} \& {(2,1).}
        	\arrow[color={rgb,255:red,214;green,92;blue,92}, curve={height=-6pt}, from=1-1, to=1-2]
        	\arrow[color={rgb,255:red,214;green,92;blue,92}, curve={height=-6pt}, from=1-2, to=1-1]
        	\arrow[color={rgb,255:red,92;green,92;blue,214}, from=1-2, to=1-1]
        	\arrow[color={rgb,255:red,214;green,92;blue,92}, curve={height=-6pt}, from=1-3, to=1-4]
        	\arrow[color={rgb,255:red,214;green,92;blue,92}, curve={height=-6pt}, from=1-4, to=1-3]
        \end{tikzcd}\]
        In particular there is just one covering component and $V$ is absolutely indecomposable.
    \end{Example}
\section{Declarations}
\subsection{Acknowledgement}
This version of the article has been accepted for publication, after peer review but is not the Version of Record and does not
reflect post-acceptance improvements, or any corrections. The Version of Record is available online at:\\
http://dx.doi.org/10.1007/s10468-025-10326-9
\subsection{Ethical Approval}
not applicable
\subsection{Funding}
not applicable
\subsection{Availability of data and materials}
not applicable
\bibliographystyle{alpha}
\bibliography{references}

\end{document}